\documentclass[11pt]{article}
\usepackage{graphicx}
\usepackage{amsmath,amsthm,amssymb,enumerate}
\usepackage{euscript,mathrsfs}
\usepackage{color}
\usepackage{dsfont}
\usepackage[notref,notcite]{showkeys}
\usepackage[left=2cm,right=2cm,top=4cm,bottom=4cm]{geometry}
\usepackage{color}
\usepackage[framemethod=tikz]{mdframed}
\usepackage{bm}
\allowdisplaybreaks

\usepackage{esint}
\usepackage{soul}

\catcode`\@=11 \@addtoreset{equation}{section}

\catcode`\@=12

\allowdisplaybreaks

\newtheorem{Theorem}{Theorem}[section]
\newtheorem{Proposition}[Theorem]{Proposition}
\newtheorem{Lemma}[Theorem]{Lemma}
\newtheorem{Corollary}[Theorem]{Corollary}

\theoremstyle{definition}
\newtheorem{Definition}[Theorem]{Definition}

\newtheorem{Remark}[Theorem]{Remark}

\newcommand{\bTheorem}[1]{
\begin{Theorem} \label{T#1} }
\newcommand{\eT}{\end{Theorem}}

\newcommand{\bProposition}[1]{
\begin{Proposition} \label{P#1}}
\newcommand{\eP}{\end{Proposition}}

\newcommand{\bLemma}[1]{
\begin{Lemma} \label{L#1} }
\newcommand{\eL}{\end{Lemma}}

\newcommand{\bCorollary}[1]{
\begin{Corollary} \label{C#1} }
\newcommand{\eC}{\end{Corollary}}

\newcommand{\bRemark}[1]{
\begin{Remark} \label{R#1} }
\newcommand{\eR}{\end{Remark}}

\newcommand{\bDefinition}[1]{
\begin{Definition} \label{D#1} }
\newcommand{\eD}{\end{Definition}}

\newcommand{\Ds}{\mathbb{D}_x}

\newcommand{\bfphi}{\boldsymbol{\varphi}}

\newcommand{\bFormula}[1]{
\begin{equation} \label{#1}}
\newcommand{\eF}{\end{equation}}

\newcommand{\Ov}[1]{\overline{#1}}

\newcommand{\DC}{C^\infty_c}
\newcommand{\aleq}{\stackrel{<}{\sim}}

\newcommand{\vu}{\vc{u}}
\newcommand{\vr}{\rho}

\newcommand{\Div}{{\rm div}_x}
\newcommand{\Grad}{\nabla_x}

\newcommand{\dx}{\,{\rm d} {x}}

\newcommand{\dt}{\,{\rm d} t }

\newcommand{\dxdt}{\dx  \dt}
\newcommand{\intO}[1]{\int_{\Omega} #1 \ \dx}

\newcommand{\D}{{\rm d}}

\newcommand{\R}{\mathbb{R}}

\newcommand{\br}{ \nonumber \\ }

\newcommand{\vc}{\mathbf}

\newcommand{\bu}{\mathbf{u}}
\newcommand{\bphi}{\bm{\varphi}}
\renewcommand{\leq}{\leqslant}
\renewcommand{\geq}{\geqslant}

\newcommand{\bfomega}{\boldsymbol{\omega}}
\def\softd{{\leavevmode\setbox1=\hbox{d}%
          \hbox to 1.05\wd1{d\kern-0.4ex{\char039}\hss}}}
\definecolor{Cgrey}{rgb}{0.85,0.85,0.85}
\definecolor{Cblue}{rgb}{0.50,0.85,0.85}
\definecolor{Cred}{rgb}{1,0,0}
\definecolor{fancy}{rgb}{0.10,0.85,0.10}

\newcommand\Cbox[2]{%
    \newbox\contentbox%
    \newbox\bkgdbox%
    \setbox\contentbox\hbox to \hsize{%
        \vtop{
            \kern\columnsep
            \hbox to \hsize{%
                \kern\columnsep%
                \advance\hsize by -2\columnsep%
                \setlength{\textwidth}{\hsize}%
                \vbox{
                    \parskip=\baselineskip
                    \parindent=0bp
                    #2
                }%
                \kern\columnsep%
            }%
            \kern\columnsep%
        }%
    }%
    \setbox\bkgdbox\vbox{
        \color{#1}
        \hrule width  \wd\contentbox %
               height \ht\contentbox %
               depth  \dp\contentbox
        \color{black}
    }%
    \wd\bkgdbox=0bp%
    \vbox{\hbox to \hsize{\box\bkgdbox\box\contentbox}}%
    \vskip\baselineskip%
}

\mdfdefinestyle{MyFrame}{%
	linecolor=black,
	outerlinewidth=1pt,
	roundcorner=5pt,
	innertopmargin=\baselineskip,
	innerbottommargin=\baselineskip,
	innerrightmargin=10pt,
	innerleftmargin=10pt,
	backgroundcolor=white!20!white}

\date{}

\makeindex
\begin{document}


\title{On the effect of a large cloud of rigid particles on the motion of an incompressible non--Newtonian fluid}

\author{Eduard Feireisl
	\thanks{The work of E.F. was partially supported by the
		Czech Sciences Foundation (GA\v CR), Grant Agreement
		24-11034S. The Institute of Mathematics of the Academy of Sciences of
		the Czech Republic is supported by RVO:67985840. The research of A.R. has been supported by the Alexander von Humboldt-Stiftung / Foundation. A.Z has been partially supported by the Basque Government through the BERC 2022-2025 program and by the Spanish State Research Agency through BCAM Severo Ochoa CEX2021-001142 and through project PID2020-114189RB-I00 funded by Agencia Estatal de Investigación (PID2020-114189RB-I00 / AEI / 10.13039/501100011033). A.Z. was also partially supported  by a grant of the Ministry of Research, Innovation and Digitization, CNCS - UEFISCDI, project number PN-III-P4-PCE-2021-0921, within PNCDI III.} 
	\and Arnab Roy$^1$ \and Arghir Zarnescu$^{2,3,4}$
}

\date{\today}

\maketitle

\bigskip

\centerline{$^*$  Institute of Mathematics of the Academy of Sciences of the Czech Republic}

\centerline{\v Zitn\' a 25, CZ-115 67 Praha 1, Czech Republic.} 

\centerline{$^1$ Technische Universit\"{a}t Darmstadt}

\centerline{Schlo\ss{}gartenstra{\ss}e 7, 64289 Darmstadt, Germany.}

\centerline{$^2$ BCAM, Basque Center for Applied Mathematics}

\centerline{Mazarredo 14, E48009 Bilbao, Bizkaia, Spain.}

\centerline{$^3$IKERBASQUE, Basque Foundation for Science, }

\centerline{Plaza Euskadi 5, 48009 Bilbao, Bizkaia, Spain.}

\centerline{$^4$`Simion Stoilow" Institute of the Romanian Academy,}

\centerline{21 Calea Grivi\c{t}ei, 010702 Bucharest, Romania.}

\maketitle

\begin{abstract}
	
We show that the collective effect of $N$ rigid bodies 
{$(\mathcal{S}_{n,N})_{n=1}^N$ of diameters $(r_{n,N})_{n=1}^N$} immersed in an incompressible non--Newtonian fluid is 
negligible in the asymptotic limit $N \to \infty$ as long as their total packing volume {$\sum_{n=1}^N r_{n,N}^d$, $d=2,3$} tends to zero 
exponentially --  ${\sum_{n=1}^N r_{n,N}^d \approx A^{-N}}$ -- for a certain constant $A > 1$. The result is rather surprising and in a sharp contrast with the 
associated homogenization problem, where the same number of obstacles can completely stop the fluid motion in the case of 
shear thickening viscosity. A large class of non--Newtonian fluids is included, for which the viscous stress is a subdifferential of a convex potential.

\end{abstract}

\bigskip


\section{Introduction}
\label{i}

We consider the motion of {$N$ rigid bodies $(\mathcal{S}_{n,N})_{n=1}^N$ of diameters $(r_{n,N})_{n=1}^N$,
\begin{equation} \label{i1}
	0 < r_{1,N} \leq r_{2,N} \leq \dots \leq r_{N,N},	
\end{equation}	
immersed in a non-Newtonian fluid confined to a bounded spatial domain $\Omega \subset \R^d$, $d=2,3$.} By non-Newtonian we mean that the viscous stress $\mathbb{S}$ is determined as a subdifferential of a convex function $F$ depending on the symmetric velocity gradient
\[
\Ds \vu = \frac{1}{2} (\Grad \vu + \Grad^t \vu).
\] 
Specifically, we consider the viscous stress $\mathbb{S}$ satisfying  
\[
\mathbb{S} \in \partial F(\Ds \vu),\ \ \mbox{where}\ \Ds \vu = \frac{1}{2} (\Grad \vu + \Grad^t \vu),\ F: R^{d \times d}_{\rm sym} \to [0, \infty) \ \mbox{convex}.
\]
{
In addition, we assume the fluid is incompressible, specifically the fluid velocity 
$\vu = \vu(t,x)$ is solenoidal,
\begin{equation} \label{i2}
	\Div\vu = 0. 
\end{equation}
}

We focus on the case of a super--quadratic growth of $F$, where the associated capacity of the family of rigid bodies is strictly positive. We show that the presence of rigid bodies has no influence on the fluid motion in the asymptotic limit {$N \to \infty$ as soon as 
\[
\sum_{n=1}^N r^d_{n,N} \leq A^{-N} 
\]
for a certain constant $A > 1$, see Theorem \ref{Tw1} below. }
  The result shows 
a substantial difference between the fluid--structure interaction problem, where the rigid objects share the same velocity with the adjacent fluid, and the homogenization limit, where the fluid velocity is fixed to be zero on static obstacles.

For the sake of simplicity, we consider the potential $F$ in the form to fit the 
$L^p-$framework. Generalizations to a larger class of potentials satisfying the so--called $\Delta^2_2$ condition are possible.  

\subsection{Fluid structure interaction problem} 
\label{FS}

{
The bodies are represented by connected compact sets $(\mathcal{S}_{n,N})_{n=1}^N$. For the sake of simplicity 
we suppose the mass density of each body is a positive constant $\vr_{n,N} > 0$, $n = 1, \dots, N$. Accordingly, 
without loss of generality, we fix
the barycentres of $\mathcal{S}_{n,N}$ to be located at the origin,  
\[
\frac{1}{|\mathcal{S}_{n,N}|} \int_{\mathcal{S}_{n,N}} \vc{x} \ dx = 0.
\] 
The positions of $\mathcal{S}_{n,N} (t)$  
at the time $t \geq 0$ are determined by a family of affine isometries
\begin{align}
( \sigma_{n,N} (t) )_{t \geq 0},\ \sigma_{n,N}(t) \vc{x} &= \vc{h}_{n,N}(t) + \mathbb{O}_{n,N}(t) \cdot \vc{x} ,\
\mathbb{O}_{n,N}(t) \in SO(d),\br
\mathcal{S}_{n,N}(t) &= \sigma_{n,N}(t) [ \mathcal{S}_{n,N}], \ n = 1, \dots, N,\ 
t \geq 0,
\nonumber
\end{align}
}
with the associated rigid velocities 
\[
\vc{w}_{n,N} (t,x) = \vc{Y}_{n,N}(t) + \mathbb{Q}_{n,N}(t) (x - \vc{h}_{n,N}(t)),\ 
\vc{Y}_{n,N} = \frac{\D }{\dt} \vc{h}_{n,N}(t),\ \mathbb{Q}_{n,N}(t) = \left[ \frac{\D }{\dt} 
\mathbb{O}_{n,N}(t) \right] \circ \mathbb{O}^{-1}_{n,N}(t). 
\]

The time evolution of the fluid velocity is determined by the linear momentum balance
\begin{equation} \label{i3}
\vr_f \Big( \partial_t \vu +  \vu \cdot \Grad \vu \Big) + \Grad \Pi = \Div \mathbb{S} + \vr_f \vc{g} ,
\end{equation}
{where $\vr_f > 0$ is the fluid density, $\Pi$ the pressure, and $\vc{g}$ a given external driving force.} We consider a general form of 
the viscous stress {$\mathbb{S}$, 
\begin{equation} \label{i4}
\mathbb{S} \in \partial F (\Ds \vu), 
	\end{equation} }
where $F$ is a convex potential defined on the space  
$R^{d \times d}_{\rm sym}$ of
symmetric tensors. More specifically, we suppose 
\begin{align} 
F: 	R^{d \times d}_{\rm sym} &\to [0, \infty) \ \mbox{convex},\ F(0) = 0, \br
\underline{F} | \mathbb{D} |^p - c &\leq F(\mathbb{D}) \leq \Ov{F} |\mathbb{D}|^p + c \ \mbox{for any} \ \mathbb{D} \in \R^{d \times d}_{\rm sym},
		\label{i5}		
 \end{align}
where $\underline{F}$, $\Ov{F}$ are positive constants. The class of potentials 
is large enough to accommodate the scale of models ranging from the standard Navier--Stokes fluid with $F = \mu |\mathbb{D}|^2$ to { the power-law fluids $F(\mathbb{D}) = (\alpha |\mathbb{D}|^2 + \beta |\mathbb{D}|^p)$, $\beta > 0$, referred to as shear-thickening if $p>2$.
Moreover, the potential $F$ need not be differentiable at $\mathbb{D} = 0$, 
which is characteristic for} Bingham fluids 
\begin{align} 
\Ds \vu = 0 \ &\Leftrightarrow \ 
\ |\mathbb{S}| \leq \Ov{S}, \br  |\mathbb{S}| \geq \Ov{S} 
\ &\Leftrightarrow \ \mathbb{S} = \Ov{S} 
\frac{\Ds \vu }{|\Ds \vu|} + 2 \mu \Ds \vu,
\nonumber
\end{align}
and other types of ``implicitly'' constituted fluids discussed by M\' alek and Rajagopal 
\cite{MAL5}, \cite{MAL1}. 

The standard fluid structure interaction problem requires continuity of the velocities and stresses on the fluid--body interface. Keeping in mind the mass densities $\vr_{n,N}$ are constant, we introduce the mass of the bodies, 
\[
m_{n,N} = \int_{\mathcal{S}_{n,N}} \vr_{n, N} \ \dx = \vr_{n,N} |\mathcal{S}_{n,N}|,  
\] 
and the barycentres  
\[
\vc{h}_{n,N}(t) = \frac{1}{|\mathcal{S}_{n,N}|} \int_{\mathcal{S}_{n,N}(t)} \vc{x} \ \dx.
\]
The inertial tensor is given by
 {
\[
[\mathbb{J}_{n,N} \cdot \vc{a}] \cdot \vc{b} = \vr_{n,N}
\int_{\mathcal{S}_{n,N} }  \Big[ \vc{a} \wedge x  \Big] \cdot \Big[ \vc{b} \wedge x \Big] \ \dx,\quad\mbox{if}\quad d=3,  
\]
and 
\begin{equation*}
\mathbb{J}_{n,N} = \vr_{n,N} \int_{\mathcal{S}_{n,N} }   |x|^2 \ \dx \quad\mbox{if}\quad d=2.
\end{equation*}
Here, the symbol $\vc{a} \wedge \vc{b}$ denotes the standard 
cross product of vectors $\vc{a}, \vc{b} \in \R^3$, while 
\begin{equation*}
	\vc{a} \wedge \vc{b} = a_1b_2-a_2b_1 \in \R^1,\quad \alpha \wedge \vc{b}=\alpha (-b_2, b_1) \in \R^2 \ \mbox{if} \ \vc{a}, \vc{b} \in \R^2.
\end{equation*}		
}

The continuity of the velocity field, or the \emph{no--slip} boundary conditions read 
\begin{equation} \label{i6}
	\vu(t, \cdot) |_{\mathcal{S}_{n,N}(t) } = \vc{w}_{n,N} (t, \cdot),\ 
	n=1, \dots, N,\ t \geq 0. 
\end{equation}
The continuity of the stresses requires 
\begin{equation} \label{i7}
m_{n,N} \frac{\D^2 \vc{h}_{n,N}}{\dt^2 }(t)	= -\int_{\partial \mathcal{S}_{n,N}(t)}
\Big( \mathbb{S} - \Pi \ \mathbb{I} \Big) \cdot \vc{n}\ \D S_{x} +  \int_{\mathcal{S}_{n,N}(t)} \vr_{n,N}\ \vc{g} 
\ \dx,\ 
n = 1, \dots, N, 
	\end{equation}
and 
\begin{align} 
\mathbb{J}_{n,N} \frac{\D }{\dt} \bfomega_{n,N}(t) &= \mathbb{J}_{n,N}(t) \bfomega_{n,N} \wedge \bfomega_{n,N} - \int_{\partial \mathcal{S}_{n,N}(t)} 
(x - \vc{h}_{n,N}(t)) \wedge (\mathbb{S} - \Pi\ \mathbb{I})\cdot \vc{n}\ \D S_x 
\br &+ \int_{\mathcal{S}_{n,N}(t)} \vr_{n,N} (x - \vc{h}_{n,N}(t)) \wedge \vc{g} 
\ \dx,\ n = 1, \dots, N,
\label{i8}
	\end{align}
where the angular velocities $\bfomega_{n,N}$ are defined through the identity
\begin{align}
	\mathbb{Q}_{n,N}(t)(x - \vc{h}_{n,N}) = \bfomega_{n,N}(t) \wedge (x - \vc{h}_{n,N}) .
	\nonumber
\end{align}

Finally, we suppose the whole system occupies a bounded regular domain $\Omega \subset \R^d$, $d=2,3$, and impose the no--slip conditions 
\begin{equation} \label{i9}
	\vu|_{\partial \Omega} = 0.
\end{equation}

\subsection{Asymptotic limit}

We consider global--in--time solutions to the fluid--structure interaction problem specified in Section \ref{FS} in the asymptotic limit of a large number of rigid bodies $N \to \infty$ and vanishing total {packing volume} 
\[
{\rm vol}[\cup_{n=1}^N\mathcal{S}_{n,N}] \aleq 
\sum_{n=1}^N r_{n,N}^d \equiv {\rm vol}[N] \to 0 \ \mbox{as}\ N \to \infty.
\]
We adapt 
the elegant framework of weak solutions introduced by Judakov \cite{Juda} , see also Gunzburger, Lee, and Seregin \cite{GLSE}. The \emph{existence} of global--in--time solutions in Judakov's class was proved in \cite{EF65} for Newtonian fluids and 
$d=3$, by San Martin, Starovoitov, and Tucsnak \cite{SST} for Newtonian fluids and $d=2$, and in 
\cite{FeHiNe} for a particular case of a non--Newtonian fluid  
corresponding to $F(\mathbb{D}) = (\alpha |\mathbb{D}|^2 + \beta |\mathbb{D}|^p)$, $\beta > 0$, $p \geq 4$, $d = 3$. Besides, there is a large number of existence results ``up to the first contact''. In the case of an increasing number of rigid bodies, however, the time of the first contact may 
be short and even approach zero in the asymptotic limit unless $p$ is large enough, see e.g. Starovoitov \cite{Staro}, \cite{STA}. 

{Intuitively, one might expect that the presence of a ``small'' moving object wouldn't significantly alter the fluid motion. Indeed, there are several results available concerning the motion of a \emph{single} rigid body in viscous Newtonian fluids. Dashti and Robinson \cite{MR2781594} consider a system involving a viscous fluid interacting with a non-rotating rigid disc, demonstrating that the presence of the disc has no impact on the flow in the asymptotic regime. Iftimie et al. \cite{MR2244381} investigated the flow around a small rigid obstacle, while Lacave \cite{MR2557320} examined the limit of a viscous fluid's flow around the exterior of a thin obstacle that contracts to a curve. Further results have been obtained by Lacave and Takahashi \cite{LacTak},  He and Iftimie \cite{HeIft1, HeIft2}, and Chipot et al. \cite{MR4066792}. More recently, Bravin and Nečasová [5] demonstrated that when the density is very high, the rigid object continues to move at its initial velocity, unaffected by the surrounding fluid. 
}

In \cite{FeiRoyZar2022}, we have shown that a large family of rigid bodies has no impact on the motion of a viscous Newtonian fluid as long as their number and the diameter are properly related. Our goal is to extend this result to the non--Newtonian case. In particular, we consider the situation 
\begin{equation} \label{i10}
	p \geq d ,
	\end{equation}
where $p$ is the exponent in \eqref{i5}. 
The result is quite surprising in particular for $p > d$. 
Indeed, for $p > d$, the velocity field $\vu$ is continuous and the total $p-$ capacity of any (even finite) number of points is positive. In the limit case $p = d$, we show the same result still allowing positive $p-$capacity of the total ensemble of rigid bodies. It is worth noting that for $d = p =2$, the Newtonian case considered already in \cite{FeiRoyZar2022} is included. Here, we improve the result of \cite{FeiRoyZar2022} by showing strong convergence of the velocity gradients.

The key ingredient of our approach is the time dependent restriction operator constructed 
in \cite{FeiRoyZar2022}, well adapted to the fluid--structure interaction problems including collisions of several rigid bodies. The new difficulty compared to the linear Newtonian case is the non--linear dependence of the stress tensor on $\Ds \vu$. This problem is overcome replacing the dissipation term $\mathbb{S}: \Ds \vu$ in the energy inequality by the expression 
\[
F(\Ds \vu) + F^* (\mathbb{S}), 
\]
where $F^*$ is the convex conjugate. The strong convergence of $\Ds \vu$ in the asymptotic limit is then obtained by the methods of convex analysis.

The paper is organized as follows. In Section \ref{w}, we introduce the weak formulation, formulate the main hypotheses, and state our main result. 
In Section \ref{r}, we recall the definition of the restriction operator and 
its basic properties adapted to the regularity class of non--Newtonian fluids. 
{The main result, concerning the asymptotic limit in the large $N$/small ${\rm vol}[N]$ regime is proved in Section \ref{a}. The paper is concluded by a short discussion in Section \ref{c}. }

\section{Weak formulation, main result}
\label{w}

Adapting the original definition of Judakov \cite{Juda}, we introduce a weak formulation of the fluid--structure interaction problem specified in Section \ref{FS}.

\subsection{Weak formulation}

Given the positions of the right bodies $\mathcal{S}_{n,N}(t)$, we introduce the fluid domain 
\[
\Omega_{f,N} (t) = \Omega \setminus \cup_{n=1}^N \mathcal{S}_{n,N}(t). 
\]
The total density of the fluid--structure system satisfies the transport equation 
\[
\partial_t \vr + \vu\cdot\Grad  \vr  = 0
\] 
in the weak sense, meaning 
\begin{equation} \label{w1}
	\int_0^T \int_{\R^d} \Big[ \vr \partial_t \varphi + \vr \vu \cdot \Grad \varphi \Big] \dx \dt = - \int_{\R^d} \vr_0 \varphi (0, \cdot) \dx
\end{equation}	
for any $\varphi \in C^1_c([0,T) \times \R^d)$. Here the velocity field
\[
\vu \in L^\infty(0,T; L^2(\Omega; \R^d)) \cap L^p(0,T; W^{1,p}_0 (\Omega; \R^d))
\]
has been set to be zero outside $\Omega$, and
\[
\vr(t, \cdot) = \left\{ \begin{array}{l} \vr_{n,N} \ \mbox{in}\ 
	\mathcal{S}_{n,N}(t), \\ 
\vr_{f}\ \mbox{if}\ x \in \Omega_{f,N}(t), \\
0 \ \mbox{in}\ \R^d \setminus \Omega.
 \end{array} \right.
\]

The velocity $\vu$ belongs to the class 
\begin{equation} \label{w2}
	\vu \in L^\infty(0,T; L^2(\Omega; \R^d)) \cap 
	L^p(0,T; W^{1,p}_0(\Omega; \R^d)), 
\end{equation}
and
\begin{equation} \label{w3}
	(\vu(t, \cdot) - \vc{w}_{n,N} (t, \cdot)) \in W^{1,p}_0( \R^d \setminus \mathcal{S}_{n,N}(t) ) \ {\mbox{for a.a.}\ t \in (0,T).}
	\end{equation}	

The momentum balance \eqref{i3} is rewritten in the form 
\begin{equation} \label{w4}
\int_0^T \int_{\R^d} \Big[ \vr \vu \cdot \partial_t \bfphi + [\vr \vu \otimes \vu]: \Ds \bfphi\Big] \dxdt = \int_0^T \int_{\R^d} \Big[ \mathbb{S} : \Grad \bfphi - \vr \vc{g} \cdot \bfphi \Big] \dxdt - \int_{\R^d} \vr_0 \vu_0{ \cdot \bfphi(0,\cdot)} \ \dx	
	\end{equation}
for any $\bfphi \in C^1_c([0,T) \times \Omega; \R^d)$, $\Div \bfphi = 0$,  
\begin{equation} \label{w5}
\Ds \bfphi (t, \cdot) = 0 \ \mbox{on an open neighbourhood of}\ \cup_{n=1}^N {\mathcal{S}_{n,N}(t)}. 
	\end{equation}
In addition, 
\begin{equation} \label{w6}
		\mathbb{S} \in \partial F(\Ds \vu) \ \mbox{a.a. in}\ (0,T) \times \Omega. 	
\end{equation}

Finally, the weak formulation includes the total energy balance 
\begin{equation} \label{w7}
\frac{1}{2} \intO{ \vr |\vu|^2 (\tau, \cdot) } + 
\int_0^\tau \intO{ \mathbb{S}: \Grad \vu }\ dt \leq 
\frac{1}{2} \intO{ \vr_0 |\vu_0|^2 ( \cdot) } + \int_0^\tau \intO{ \vr\vc{g} \cdot \vu } \dt 
	\end{equation}
for a.a. $\tau \in (0,T)$.

\subsection{Main result}

Having collected the preliminary material, we are ready to formulate our main result. 

	
\begin{Theorem} \label{Tw1}	

Let $\Omega \subset \R^d$, $d=2,3$ be a bounded domain of class $C^{2 + \nu}$.
Suppose the following hypotheses are satisfied:
\begin{itemize}
 \item
The potential $F$ is convex satisfying \eqref{i5}, where $p \geq d$. 
\item
The volume force $\vc{g}$ belongs to the class 
$\vc{g} \in L^\infty(\Omega; \R^d)$.
\item For a positive integer $N$, the initial configuration of the rigid bodies is represented 
by a family of {connected} compact sets $(\mathcal{S}_{n,N}(0))_{n=1}^N$ with diameters $(r_{n,N})_{n=1}^N$, 
\[
0 < r_{1,N} \leq \dots \leq r_{N,N}. 
\]
	
\item The (initial) mass density distribution 
\[
\vr_{0,N} = 
\mathds{1}_{\Omega_{f,N}(0)} \vr_{f}  +\sum_{n=1}^N \mathds{1}_{\mathcal{S}_{n,N}(0)} \vr_{n,N}
\]
satisfies  $\vr_{f}> 0,\ \vr_{n,N} > 0,\ n=1, \dots, N$, and
\begin{equation} \label{w9}
\| \vr_{0,N} \|^q_{L^q(\Omega)} = \vr_f^q |\Omega_{f,N} | + \sum_{n=1}^N \vr_{n,N}^q |\mathcal{S}_{n,N} | 
\aleq 1 \ \mbox{for some}\ q > 1
\end{equation}
uniformly for $N \to \infty$.

\item 
The rigid bodies are of positive volume (the d-dimensional Lebesgue measure), 
\begin{equation} \label{w10a}
	|\mathcal{S}_{n,N} | > 0 ,\ \mbox{with barycentres}\ \vc{h}_{n,N}(t) = \frac{1}{|\mathcal{S}_{n,N} |}\int_{\mathcal{S}_{n,N}(t)} \vc{x} \dx,\ n=1, \dots, N, 
	{\ t \geq 0.}
\end{equation} 
In addition, if $p = d$, we suppose there exists $\lambda > 0$ such that 
\begin{equation} \label{w10}
\lambda r_{n,N}^\beta \leq |\mathcal{S}_{n,N} | ,\ n = 1, \dots, N	
	\end{equation}  	
for some $\beta \geq d$.

\item The initial {data $(\vr_{0,N}, \vu_{0,N})_{N=1}^\infty$ satisfy
\begin{align} 
\vr_{0,N} &\to \Ov{\vr} \ \mbox{in}\ L^1(\Omega),\ \Ov{\vr} > 0 \ \mbox{constant}, \label{w11a} \\	
\vr_{0,N} \vu_{0,N} &\to \Ov{\vr} \vu_0 \ \mbox{weakly in}\ L^2(\Omega; \R^d), \label{w11}, \\
\intO{ \vr_{0,N} |\vu_{0,N}|^2 } &\to \intO{ \Ov{\vr} |\vu_0|^2 } \label{w11b}
\end{align}	
as $N \to \infty$.}  
  
\end{itemize}
	
Then there is a universal constant $A = A(p,q)$ such that (up to a subsequence) 
\begin{equation} \label{w12}
\vu_N \to \vu \ \mbox{weakly in}\ L^p(0,T; W^{1,p}_0(\Omega;\R^d)) 
\ \mbox{and strongly in}\ L^2((0,T) \times \Omega; \R^{d}), 	 
\end{equation}
where $\vu$ is a weak solution of the system 
\begin{align}
\Div \vu &= 0, \br	 
\Ov{\vr} \partial_t \vu + \Ov{\vr} \Div (\vu \otimes \vu) + \Grad \Pi &= \Div \mathbb{S} + \Ov{\vr} \vc{g}, 
\ \mathbb{S} \in \partial F(\Ds \vu),	\br
\vu(0, \cdot) &= \vu_0,
	\label{w13}
	\end{align}
whenever 
\begin{equation} \label{w14}
A^N{ {\rm vol}[N] }\to 0 \ \mbox{as}
 \ N \to \infty,\ \mbox{ where }{ {\rm vol}[N] \equiv \sum_{n=1}^N r_{n,N}^d.}
\end{equation}

If, in addition, the function $F$ is strictly convex, then 
\begin{equation} \label{w15}
	\vu_N \to \vu \ \mbox{strongly in}\ L^p(0,T; W^{1,p}_0(\Omega; \R^d)).		
\end{equation}	
\end{Theorem}	
	
	
\begin{Remark} \label{wR1}
In many physically relevant examples of the potential $F$, the limit problem \eqref{w13} admits a unique weak solution, see e.g. Breit \cite{BreitB} ,  
Lady\v{z}enskaja \cite{Lady67}, 
M\' alek et al. \cite{BlMaRa}, 
\cite{BMR1}, \cite{MNRR}, among others. In this case, the limit is unconditional, meaning there is no need to pass to subsequences.  
	\end{Remark}
	
{
	\begin{Remark} \label{wwrr}
In view of \eqref{w9}, \eqref{w14}, it is easy to see that $\Ov{\vr} = \vr_f$.		
		\end{Remark}		
}	
\begin{Remark} \label{wR1a}
{Under the hypotheses of Theorem \ref{Tw1}, there exists a family of open balls, 
\[
B_{r_{n,N}}(\vc{h}_{n,N}(t)) \equiv \left\{ x \in \R^d \ \Big|\ |x - \vc{h}_{n,N}(t) | < r_{n,N} \right\}, 		
\]
such that
\begin{align} 
	\mathcal{S}_{n,N}(t) \subset B_{r_{n,N}}(\vc{h}_{n,N}(t))\ n=1, \dots, N,\ t \geq 0.
	\label{w8}	
\end{align} 
Indeed, by virtue of the sharp form of Jensen's inequality, 
\begin{align}
|\vc{y} - \vc{h}_{n,N}(t)|^2 =
\left| \frac{1}{|\mathcal{S}_{n,N}(t)|} \int_{\mathcal{S}_{n,N}(t)} (\vc{y} - x ) \dx    \right|^2 < \frac{1}{|\mathcal{S}_{n,N}(t)|} \int_{\mathcal{S}_{n,N}(t)} |\vc{y} - x |^2 \dx \leq r_{n,N}^2
\nonumber
\end{align}
whenever $\vc{y} \in \mathcal{S}_{n,N}$. Accordingly, the quantity
\[
{\rm vol}[N] = \sum_{n=1}^N r_{n,N}^d 
\]		
can be seen as a packing volume with respect to the balls $B_{r_{n,N}}(\vc{h}_{n,N})$,
and, in accordance with the hypotheses \eqref{w10a}, \eqref{w10}, can be considerably larger then the total volume of the rigid bodies 
$\sum_{n=1}^N |\mathcal{S}_{n,N}|$. Nonetheless, these two quantities can be of the same order if the rigid bodies are balls centred at $\vc{h}_{n,N}$.}
	\end{Remark}	

\begin{Remark} \label{wR2}
If $d = p =2$, the result includes the Newtonian case $F(\mathbb{D}) = \mu |\mathbb{D}|^2$. It is worth noting that the same scaling 
as in \eqref{w14} yields a different \emph{homogenization} limit with a Brinkmann friction term in the limit system, see e.g. Allaire \cite{Allai4}.
{Moreover, the strong convergence claimed in \eqref{w15} fails in the homogenization limit.}
\end{Remark}

\begin{Remark} \label{wR3}
{If $p>d$, 
the velocity fields $\vu_N$ belonging to the Sobolev space	
$W^{1,p}$ are uniformly continuous for any fixed time $t \geq 0$. In particular, the $p-$capacity of a \emph{single} point is positive. Consequently, any homogenization process requiring the velocity to vanish on each rigid 
body would necessarily give rise to $\vu_N \to 0$, 
meaning the fluid motion would stop in the asymptotic limit. 
}

\end{Remark}

\section{Restriction operator}
\label{r}

Selecting an appropriate restriction operator holds paramount importance in our analysis. Diverging from the predominant approach found in the existing literature, which typically involves modifying test functions to vanish on the body, we leverage the flexibility granted by condition \eqref{w5}. Instead, we substitute the function on a ball with its integral average over that same ball. This approach has previously been employed in \cite[Section 4.1]{FeiRoyZar2022}, where comprehensive proofs of the following assertions can be found. 

Consider a function \begin{align} 
H &\in C^\infty(\mathbb{R}), \ 0 \leq H(Z) \leq 1,\ H'(Z) = H'(1- Z) \ \mbox{for all}\ Z \in \mathbb{R}, \br
H(Z) &= 0 \ \mbox{for} \ - \infty < Z \leq \frac{1}{4},\ 
H(Z) = 1 \ \mbox{for}\ \frac{3}{4} \leq Z < \infty
\nonumber
\end{align}

{For $\varphi \in {L^1_{\rm loc}(\R^d)}$ , $r > 0$ we define  $E_r$,}
\begin{align}\label{def:E}
E_r [\varphi] (x) = 
\frac{1}{|B_r(0)  |} \int_{B_r(0)} \varphi \ dz \ H \left( 2 - \frac{|x|}{r} \right)
+ \varphi(x) H \left( \frac{ | x |}{r} - 1 \right).
\end{align}
{where $B_r(0) $ denotes the ball centred at zero with the radius $r > 0$.} Obviously, the operator 
$E_r$ maps the space $\DC(\R^d)$ into itself, and, moreover 
\begin{align} 
	\| E_r [\varphi] \|_{L^p(\R^d)} &\aleq \| \varphi \|_{L^p(\R^d)}, 
	\label{T3} \\ 
\| \Grad E_r [\varphi] \|_{L^p(\R^d; \R^d)} &\aleq \| \Grad \varphi \|_{L^p(\R^d; \R^d)}
\ \mbox{for any}\ 1 \leq p \leq \infty,	
\label{T4}
	\end{align}
uniformly for $0 < r \leq 1$, see \cite[Section 4.1]{FeiRoyZar2022}.

The operator $E_r$ fails to maintain solenoidality when applied componentwise to a solenoidal function. To address this issue, we introduce a new operator:
\begin{equation} \label{E6}
	{\vc{R}_r} [\bphi ] = E_r [\bphi] - \mathcal{B}_{2 r, r} \Big[ 
	\Div E_r [\bphi]|_{B_{2r} \setminus {B}_{r}} \Big],
	\end{equation}
where $\mathcal{B}_{2 r,  r}$ is a suitable branch of the inverse of the divergence operator defined on the 
annulus $B_{2r}(0) \setminus B_{r}(0)$. One potential formulation of $\mathcal{B}$ was initially suggested by Bogovskii \cite{BOG}, further developed by Galdi \cite{GALN}, and subsequently expanded upon by Diening et al. \cite{DieRuzSch}, Gei{\ss}ert et al. \cite{GEHEHI} among others. In our context, the properties of the operator $\mathcal{B}_{2r,r}$ can be derived through a scaling argument as explained in \cite[Section 4.1]{FeiRoyZar2022}.
	
{For $\vc{h} \in \R^d$,  we set}	 
\begin{equation} \label{E10}
	\vc{R}_r (\vc{h} )[\bphi] = S_{- \vc{h}} \vc{R}_r \Big[ S_{\vc{h}} [\bphi] \Big],
	\end{equation}
where $S_{\vc{h}}$ is the shift operator given by 
\[
S_{\vc{h}} [f] (x) = f ( \vc{h} + x ).
\]

{The basic properties of the operator $\vc{R}_r(\vc{h})$ are summarized below, 
	cf. also \cite[Proposition 5.1] {FeiRoyZar2022}. }

\begin{Proposition} \label{EP1}
	
	The operator $\vc{R}_r (\vc{h})$ is well defined for any function $\bfphi$ in the class  
	\[
	\bfphi  \in \DC(\R^d; \R^d), \ \Div \bfphi = 0
	\]
and can be uniquely extended to functions 
\[
\bfphi \in L^p_{\rm loc}(\R^d; \R^d),\ \Div \bfphi = 0 \ a.a.
\]	

Moreover, the following holds:

\begin{itemize}
	
	\item
	
\begin{equation} \label{reg1}
	\bfphi \in C^\infty (\R^d; \R^d) \ \Rightarrow \ \vc{R}_r (\vc{h})[\bfphi] \in C^\infty(\R^d; \R^d);  
\end{equation}	

\item 
\begin{equation} \label{E13}
	\Div \vc{R}_r (\vc{h})[\bphi] = \Div \bphi = 0 ;
\end{equation}	
	 
	\item 
\begin{equation} \label{E12}
\vc{R}_r (\vc{h})[\bphi] = \left\{ \begin{array}{l} \frac{1}{|{B}_{r}(\vc{h})|} 
\int_{{B}_{r}(\vc{h})} \bphi \dx \ \mbox{if}\ |x - \vc{h}| < r,  \\ 
\\ 
\bphi (x) \ \mbox{if}\ |x - \vc{h}| > 2 r;  \end{array}	\right. 
	\end{equation} 

\item 

\begin{equation} \label{E12a}
\vc{R}_r (\vc{h})	[\bfphi] = \bfphi \ \mbox{whenever}\ 
\bfphi \ \mbox{is a constant vector on}\ B_{2r}(\vc{h});
	\end{equation}

\item

\begin{equation} \label{E14bis}
	\left\| \vc{R}_r (\vc{h})[\bphi] \right\|_{L^p(\R^d; \R^{d})} \aleq 
	\| \bphi\|_{L^{p}(\R^d; \R^{d})}
\end{equation}
\begin{equation} \label{E14}
\left\| \Grad \vc{R}_r (\vc{h})[\bphi] \right\|_{L^p(\R^d; \R^{d \times d})} \aleq 
\| \Grad \bphi\|_{L^{p}(\R^d; \R^{d\times d})}
\end{equation} 
for any $1 < p < \infty$ independently of $0 < r \leq 1$;
\item  If $\bphi$ is compactly supported, then so is {$\vc{R}_r (\vc{h})[\bphi]$}. 
Specifically, 
 \begin{equation} \label{E15}
{\rm supp}[ \vc{R}_r (\vc{h})[\bphi] ] \subset \Ov{ \mathcal{U}_{2r}[ {\rm supp}[\bfphi ]] }
\end{equation}
where $\mathcal{U}_{2r}(O)$ denotes the $2r-$neighbourhood of a set $O$.
\item

\begin{equation} \label{E14bisnew}
	\left\| \vc{R}_r (\vc{h})[\bphi] -\bphi \right\|_{L^p(\R^d; \R^{d})} \leq c(p) 
	\| \bphi\|_{L^{p}(B_{2r}(\vc{h}); \R^{d})}
\end{equation}
\begin{equation} \label{E14new}
\left\| \Grad \vc{R}_r (\vc{h})[\bphi] -\Grad\bphi\right\|_{L^p(\R^d; \R^{d \times d})} \leq  c(p)
\| \Grad \bphi\|_{L^{p}(B_{2r}(\vc{h}); \R^{d\times d})}
\end{equation} 
for any $1 < p < \infty$ independently of $0 < r \leq 1$.
\end{itemize}
\end{Proposition}

 \begin{Remark}
Strictly speaking, the relations \eqref{E14bisnew}--\eqref{E14new} formulated in \cite[Proposition 5.1] {FeiRoyZar2022}
contain the norm $L^p(\R^d)$ on the right--hand side. Their local form stated here
follows directly from the construction of the operator $\vc{R}_r (\vc{h})$.
\end{Remark}

Finally, we evaluate the differential of $\nabla_{\vc{h}} \vc{R}_r (\vc{h})[\bfphi]$ for 
	a given function $\bfphi$: 	
\begin{equation} \label{formula}
\nabla_{\vc{h}} \vc{R_r}(\vc{h}) [\bfphi ] = \Grad \left( \vc{R}_r (\vc{h}) [\bfphi] \right) - 
\vc{R}_r (\vc{h}) [ \Grad \bfphi ] .	
	\end{equation}
Note carefully that if $\bfphi$ is solenoidal, meaning $\Div\bfphi = 0$, then so is 
$\Grad \bfphi$ (component--wise) and the right--hand side of \eqref{formula} is well defined.
Using the bounds \eqref{E14bis}, \eqref{E14}, the formula \eqref{formula} can be extended to solenoidal functions 
$\bfphi \in W^{1,p}_{\rm loc}(\R^d; \R^d)$ by density argument. {Finally, note that
\begin{equation} \label{formula1}
\nabla_{\vc{h}} \vc{R_r}(\vc{h}) [\bfphi ] = 0 \ \mbox{in}\ \R^d \setminus B_{2r}(\vc{h}).
\end{equation}
}
\subsection{Composition}
Recall that the radii $r_{n,N}$, $n=1,\dots, N$ of $N$ rigid bodies satisfy
\begin{equation*}
0 < r_{1,N} \leq r_{2,N} \leq \dots \leq r_{N,N}.
\end{equation*}
As we handle the {several body problem where mutual collisions are not excluded}, it is convenient to consider the composition of restriction operators
\begin{equation} \label{CC1}
\vc{R}_N (\vc{h}_{1,N}, \cdots \vc{h}_{N,N}) [\bfphi]  
= \vc{R}_{r_{1,N}} (\vc{h}_{1,N} ) \circ \vc{R}_{5r_{2,N}} (\vc{h}_{2,N}) \circ \cdots \circ \vc{R}_{5^{N-1}r_{N,N}} (\vc{h}_{N,N}) [\bfphi]	
	\end{equation}
for arbitrary $(\vc{h}_{1,N}, \cdots, \vc{h}_{N,N})$. In view of the property \eqref{E13}, the operator 
$\vc{R}_N (\vc{h}_{1,N}, \cdots \vc{h}_{N,N}) [\bfphi]$ is well defined for any solenoidal $\bfphi$ and 
\[
\Div \vc{R}_N (\vc{h}_{1,N}, \cdots \vc{h}_{N,N}) [\bfphi] = 0.
\] 

{The following crucial result is a slight modification of \cite[Lemma 4.2]{FeiRoyZar2022}.}

\begin{Lemma} \label{Lr1}
Suppose $\vc{h}_{1,N}, \dots, \vc{h}_{1,N}$ are points in $\R^d$, and 
\[
0 < r_{1,N} \leq r_{2,N} \leq \dots, \leq r_{N,N} 
\]
positive radii. Then 
\begin{equation} \label{r1}
	\vc{R}_N (\vc{h}_{1,N}, \cdots \vc{h}_{N,N}) = \Lambda_{n} \ \mbox{on}\ B_{r_{n,N}}(\vc{h}_{n,N}),\ n = 1, \dots, N,  
	\end{equation}	
where $\Lambda_n$, $n=1,\dots, N$ are constants. 	
	
	\end{Lemma}
	
\begin{proof}
The case $N=1$ is obvious. We proceed by induction. Our induction hypothesis is that \eqref{r1} holds for $N-1$, which means
\begin{equation} \label{r1a}
	\vc{R}_{r_{1,N}} (\vc{h}_{1,N} ) \circ \vc{R}_{5r_{2,N}} (\vc{h}_{2,N}) \circ \cdots \circ \vc{R}_{5^{N-2}r_{N-1,N}} (\vc{h}_{N-1,N})= \Lambda_{i} \ \mbox{on}\ B_{r_i}(\vc{h}_{i,N}),\ i = 1, \dots, N-1.  
	\end{equation}
 It is equivalent to
 \begin{equation} \label{r1b}
	 \vc{R}_{5r_{2,N}} (\vc{h}_{2,N}) \circ \cdots \circ \vc{R}_{5^{N-1}r_{N,N}} (\vc{h}_{N,N})= \Lambda_{i} \ \mbox{on}\ B_{5r_i}(\vc{h}_{i,N}),\ i = 2, \dots, N.  
	\end{equation}	
We consider two complementary cases:

\noindent
{\bf Case  (i):}
 
\begin{equation}\label{case1}
|\vc{h}_{1,N} - \vc{h}_{i,N}| + 2 r_{1,N} \leq 5 r_{i,N} \ \mbox{for some}\ i \in \{ 2, \dots, N\}.
\end{equation}
In view of the induction hypothesis \eqref{r1b} and the relation \eqref{case1}, we obtain
\begin{align} \label{r2}
 \vc{R}_{5 r_{2,N}} (\vc{h}_{2,N}) \circ \dots 
\circ \vc{R}_{5^{N-1} r_{N,N}}(\vc{h}_{N,N}) = \Lambda_i \ \mbox{on}\ B_{2 r_{1,N}}(\vc{h}_{1,N}) \ \mbox{for some}\ i \in \{ 2, \dots, N\}.
\end{align}
The relation \eqref{r2} and properties \eqref{E12}, \eqref{E12a} imply
\[
\vc{R}_{r_{1,N}}(\vc{h}_{1,N}) \circ \vc{R}_{5 r_{2,N}} (\vc{h}_{2,N}) \circ \dots 
\circ \vc{R}_{5^{N-1} r_{N,N}}(\vc{h}_{N,N}) = \vc{R}_{5 r_{2,N}} (\vc{h}_{2,N}) \circ \dots 
\circ \vc{R}_{5^{N-1} r_{N,N}}(\vc{h}_{N,N}), 
\]
and the desired conclusion \eqref{r1} follows from \eqref{r2}.

\noindent
{\bf Case  (ii):}
\[
|\vc{h}_{1,N} - \vc{h}_{i,N}| + 2 r_{1,N} > 5 r_{i,N} \ \mbox{for all}\ i = 2, \dots, N.
\]
This yields 
\[
|\vc{h}_{1,N} - \vc{h}_{i,N}| > r_{i,N} + 2 r_{1,N}, 
\]
and, consequently by the property \eqref{E12a} and induction hypothesis \eqref{r1b}, we have
\[
\vc{R}_{r_{1,N}}(\vc{h}_{1,N}) \circ \vc{R}_{5 r_{2,N}} (\vc{h}_{2,N}) \circ \dots 
\circ \vc{R}_{5^{N-1} r_{N,N}}(\vc{h}_{N,N}) = \vc{R}_{5 r_{2,N}} (\vc{h}_{2,N}) \circ \dots 
\circ \vc{R}_{5^{N-1} r_{N,N}}(\vc{h}_{N,N}) = \Lambda_i, 
\]
{on} $B_{r_{i,N}}(\vc{h}_{i,N})$ for any $i=2,\dots, N$.
	
	\end{proof}	

Finally, we use Lemma \ref{Lr1} together with \cite[Proposition 4.3]{FeiRoyZar2022} to deduce the following properties of the operator $\vc{R}_N (\vc{h}_{1,N}, \cdots \vc{h}_{N,N})$.

\begin{Proposition} \label{EP1R}
	
	The operator $\vc{R}_N (\vc{h}_{1,N}, \cdots \vc{h}_{N,N})$ is well defined for any function $\bfphi$ in the class  
	\[
	\bfphi  \in \DC(\R^d; \R^d), \ \Div \bfphi = 0
	\]
	and can be uniquely extended to functions 
	\[
	\bfphi \in L^p_{\rm loc}(\R^d; \R^d),\ \Div \bfphi = 0 \ a.a.
	\]	
	
	Moreover, the following holds:
	
	\begin{itemize}
		
		\item
		
		\begin{equation} \label{reg1R}
			\bfphi \in C^\infty (\R^d; \R^d) \ \Rightarrow \ \vc{R}_N (\vc{h}_{1,N}, \cdots \vc{h}_{N,N})[\bphi] \in C^\infty(\R^d; \R^d);  
		\end{equation}	
		
		\item 
		\begin{equation} \label{E13R}
			\Div \vc{R}_N (\vc{h}_{1,N}, \cdots \vc{h}_{N,N}) [\bphi] = \Div \bphi = 0 ;
		\end{equation}	
		
		\item 
		\begin{equation} \label{E12R}
			\vc{R}_N (\vc{h}_{1,N}, \cdots \vc{h}_{N,N}) [\bphi] = \Lambda_n - \mbox{a constant vector on} 
		\ B_{r_{n,N}} (\vc{h}_{n,N}),\ n = 1,\cdots, N;	
		\end{equation}

		\item 
		
		\begin{equation} \label{E12RR}
			\vc{R}_N (\vc{h}_{1,N}, \cdots \vc{h}_{N,N}) [\bphi] (x) = \bfphi (x) \ \mbox{whenever}\ 
			x \in \R^d \setminus \cup_{n=1}^N B_{2 ( 5^{n-1} r_{n,N} ) } (\vc{h}_{n,N}) ;
		\end{equation}
		
		\item
		
		\begin{equation} \label{E14bisR}
			\left\| \vc{R}_N (\vc{h}_{1,N}, \cdots \vc{h}_{N,N})[\bphi] \right\|_{L^p(\R^d; \R^{d})} \leq c(p)^N
			\| \bphi\|_{L^{p}(\R^d; \R^{d})}
		\end{equation}
		\begin{equation} \label{E14R}
			\left\| \Grad \vc{R}_N (\vc{h}_{1,N}, \cdots \vc{h}_{N,N})[\bphi] \right\|_{L^p(\R^d; \R^{d \times d})} \leq
			c(p)^N 
			\| \Grad \bphi\|_{L^{p}(\R^d; \R^{d\times d})}
		\end{equation} 
		for any $1 < p < \infty$;
		\item  If $\bphi$ is compactly supported, then so is {$\vc{R}_N (\vc{h}_{1,N}, \cdots \vc{h}_{N,N})[\bphi]$}. 
		Specifically, 
		\begin{equation} \label{E15R}
			{\rm supp}[ \vc{R}_N (\vc{h}_{1,N}, \cdots \vc{h}_{N,N})[\bphi] ] \subset \Ov{ \mathcal{U}_{2( 5^{N} r_{N,N})}[ {\rm supp}[\bfphi ]] }.
		\end{equation}
\item

\begin{equation} \label{com:E14bisnew}
	\left\| \vc{R}_N (\vc{h}_{1,N}, \cdots \vc{h}_{N,N})[\bphi] -\bphi \right\|_{L^p(\R^d; \R^{d})} \leq c(p)^N 
	\| \bphi\|_{L^{p}(\cup_{i=n}^N B_{2 (5^{n-1} r_{n,N})} (\vc{h}_{n,N}); \R^{d})}
\end{equation}
\begin{equation} \label{com:E14new}
\left\| \Grad \vc{R}_N (\vc{h}_{1,N}, \cdots \vc{h}_{N,N})[\bphi] -\Grad\bphi\right\|_{L^p(\R^d; \R^{d \times d})} \leq  c(p)^N
\| \Grad \bphi\|_{L^{p}(\cup_{i=n}^N B_{2(5^{n-1} r_{n,N})} (\vc{h}_{n,N}); \R^{d\times d})}
\end{equation} 
for any $1 < p < \infty$.
	\end{itemize}
\end{Proposition}

\begin{Remark}
The constant $c(p)$, depending on the norm of the Bogovskii operator, is one of the factors determining the specific value of the quantity $A$ in  	
\eqref{w14}.
	\end{Remark}

\section{Asymptotic limit: Proof of Theorem \ref{Tw1}}
\label{a}

Our ultimate goal is to let $N \to \infty$. {First, we derive suitable bounds uniform for $N \to \infty$.}
The velocity $\vu_N$ satisfies the energy inequality 
\begin{equation} \label{s2}
\frac{1}{2} \intO{ \vr_N |\vu_N|^2(\tau, \cdot) } + 
\int_0^\tau \intO{ \mathbb{S}_N : \Ds \vu_N } \leq \frac{1}{2} \intO{ \vr_{0,N} 
	|\vu_{0, N}|^2 } + \int_0^\tau \intO{ \vr_N \vc{g} \cdot \vu_N } \dt 
\end{equation}
for a.a. $\tau \in (0,T)$. 
We recall Fenchel--Young inequality 
\begin{align} 
\mathbb{S} : \mathbb{D} &\leq F(\mathbb{D}) + F^* (\mathbb{S}), \br 
\mathbb{S} : \mathbb{D} &= F(\mathbb{D}) + F^* (\mathbb{S}) \ 
\Leftrightarrow \ \mathbb{S} \in \partial F(\Ds \vu). 
\label{s3}
\end{align}
Since $\mathbb{S}_N \in \partial F(\Ds \vu_N)$, we can rewrite the energy inequality 
\eqref{s2} in the form {
\begin{align} 
	\frac{1}{2} \intO{ \vr_N |\vu_N|^2(\tau, \cdot) } &+ 
	\int_0^\tau \intO{ \Big[ F(\Ds \vu_N) + F^* (\mathbb{S}_N)\Big] }\ dt \br &\leq \frac{1}{2} \intO{ \vr_{0,N} 
		|\vu_{0, N}|^2 } + \int_0^\tau \intO{ \vr_N \vc{g} \cdot \vu_N } \dt \br 
		&\leq  \frac{1}{2} \intO{ \vr_{0,N} 
			|\vu_{0, N}|^2 } + \frac{1}{2} \int_0^\tau \intO{ \vr_N |\vc{g}|^2 } \dt  + 
			\frac{1}{2} \int_0^\tau \intO{ \vr_N |\vu_N|^2 } \dt.
		\label{s4} 
\end{align}
}

{
Next, we apply DiPerna--Lions theory \cite{DL} to the weak formulation of the equation of continuity \eqref{w1} to obtain 
\begin{equation} \label{s3a}
\intO{ G (\vr_N) (\tau, \cdot) } = \intO{ G(\vr_{0,N})} ,\ \tau > 0, 
\end{equation}	
for any continuous $G$ as soon as the integral on the right--hand side is finite. In particular, it follows from hypotheses 
\eqref{w9} and \eqref{w11a} that 
\begin{equation} \label{rhh}
	{ \vr_N \ \mbox{is bounded in}\ L^\infty(0,T; L^q(\Omega)),	}
\end{equation}
and
\begin{equation} \label{con:rho}
\vr_N \to \Ov{\vr} \ \mbox{weakly-(*) in}\ L^\infty(0,T; L^q (\Omega)) \ \mbox{and (strongly) in}\ C([0,T]; L^r(\Omega)), 
\ 1 \leq r < q.	
\end{equation}	
}

Finally, as $F$ satisfies the growth restriction \eqref{i5}, we obtain from \eqref{s4} {and \eqref{w1}:}
\begin{equation}\label{rhou}
\sqrt{\vr_N}\vu_N \mbox{ is bounded in }L^{\infty}(0,T;L^2(\Omega; \R^d)),
\end{equation}
\begin{equation}\label{bdd:uN}
\vu_N \mbox{ is bounded in } L^p(0,T;W^{1,p}_0(\Omega; \R^d)),
\end{equation}
\begin{equation*}
\mathbb{S}_N \mbox{ is bounded in } L^{p'}(0,T;L^{p'}(\Omega; \R^d)).
\end{equation*}

Thus there exists $\bu\in L^{\infty}(0,T;L^2(\Omega; \R^d)) \cap L^p(0,T;W^{1,p}_0(\Omega; \R^d))$, {
\[
\Div \vu = 0 \ \mbox{a.a. in}\ (0,T) \times \Omega,	
\]	
such that, up to a subsequence,}
\begin{equation}\label{con:ws}
 \sqrt{\vr_N}\vu_N \rightarrow \sqrt{\Ov{\vr}}\ \bu \mbox{ weakly-${*}$ in }L^{\infty}(0,T;L^2(\Omega; \R^d)),
 \end{equation}
 \begin{equation}\label{con:w}
 \vu_N \rightarrow \bu \mbox{ weakly in }L^p(0,T;W^{1,p}_0(\Omega; \R^d)),
 \end{equation}
  \begin{equation}\label{con:S}
\mathbb{S}_N \rightarrow \mathbb{S} \mbox{ weakly in }L^{p'}(0,T;L^{p'}(\Omega; \R^{d \times d})).
 \end{equation}
\medskip 

\subsection{Limit in the momentum equation}

Let  $\bfphi \in \DC ([0,T) \times \Omega; \R^d)$, $\Div \bfphi = 0$ be a smooth solenoidal function. 
Our goal is to plug the quantity $\vc{R}_N ( \vc{h}_{1,N}, \cdots, \vc{h}_{N,N})[ \bfphi ]$ as a test function in the  momentum balance
\begin{equation} \label{w4N}
\int_0^T \int_{\R^d} \Big[ \vr_N \vu_N \cdot \partial_t \bfphi + [\vr_N \vu_N \otimes \vu_N]: \Ds \bfphi\Big] \dxdt = \int_0^T \int_{\R^d} \Big[ \mathbb{S}_N : \Grad \bfphi - \vr_N \vc{g} \cdot \bfphi \Big] \dxdt - \int_{\R^d} \vr_{0,N} \vu_{0,N} \cdot \bfphi \dx 	,
\end{equation}
and perform the limit $N\to\infty$.

\subsubsection{Viscous stress}

 We want to estimate
$\| ( \bfphi - \vc{R}_N ( \vc{h}_{1,N}, \cdots, \vc{h}_{N,N})[ \bfphi ] )(t, \cdot) \|_{W^{1,p}(\R^d)}$. By virtue of the error estimates \eqref{com:E14bisnew}--\eqref{com:E14new}, we obtain 
\begin{multline}\label{errorW1p}
\left\| \vc{R}_N (\vc{h}_{1,N}, \cdots \vc{h}_{N,N})[\bphi] -\bphi\right\|^{p}_{W^{1,p}(\R^d)} \leq  c(p)^{pN}
\|\bphi\|^{p}_{W^{1,p}(\cup_{n=1}^N B_{2 (5^{n-1} r_{n,N})} (\vc{h}_{n,N}))} \\
 \leq c(p)^{pN}\|\bfphi\|^{p}_{C^{1}(\Omega)}\sum\limits_{n=1}^N |B_{2( 5^{n-1} r_{n,N})}|
 \aleq \|\bfphi\|^{p}_{C^{1}(\Omega)}A^N_1 \sum_{n=1}^N r_{n,N}^d,
\end{multline}
where {$A_1= A_1(p) = \max\{c(p)^{p}, 5^d \}$}. {Thus the choice $A = A_1(p)$ in hypothesis \eqref{w14} yields} 
\begin{equation} \label {error1}
	\left\| \vc{R}_N (\vc{h}_{1,N}, \cdots \vc{h}_{N,N})[\bphi] -\bphi\right\|_{W^{1,p}(\Omega, \R^d)} \to 0 \ \mbox{as}\ N\to \infty,
	\ \mbox{uniformly for}\ t \in [0,T].
\end{equation} 
In view of \eqref{con:S} and  \eqref{error1}, we may infer that 
\begin{equation}\label{I1}
	\int\limits_{0}^{T} \int\limits_{\Omega} \mathbb{S}_N: \Grad(\vc{R}_N ( \vc{h}_{1,N}, \cdots, \vc{h}_{N,N})[ \bfphi ]) \dx \dt
	\to \int_0^T \intO{ \mathbb{S}:  \Grad \bfphi } \dt
\end{equation}
for any $\bfphi \in C^1_c([0,T) \times \Omega; \R^d)$, $\Div \bfphi = 0$.

\subsubsection{Convective term}

{In view of the convergence \eqref{con:rho}, the uniform bounds \eqref{rhou}, \eqref{bdd:uN}, and the Sobolev embedding 
$W^{1,p} \hookrightarrow L^s$, $s \geq 1$ arbitrary finite, we get
\[
(\vr_N - \Ov{\vr}) (\vu_N \otimes \vu_N) \to 0 \ \mbox{in}\ L^r((0,T) \times \Omega; \R^{d \times d}) \ \mbox{for a certain}\ r = r(q) > 1.
\]
Revisiting the error estimates \eqref{errorW1p} we may infer that 
\[
\int_0^T \intO{ (\vr_N - \Ov{\vr}) (\vu_N \otimes \vu_N) : \Grad (\vc{R}_N ( \vc{h}_{1,N}, \cdots, \vc{h}_{N,N})[ \bfphi ]) } \dt \to 0
\ \mbox{as}\ N \to \infty 
\] 
provided the quantity $A$ in hypothesis \eqref{w14} is chosen as
\[
A = A(p,q) = \max \{ A_1(p), A_1(r') \},\ r' = r'(q),\ \frac{1}{r} + \frac{1}{r'} \leq 1. 
\]
In particular, we may suppose
\[
\vr_N \vu_N \otimes \vu_N =\sqrt{\vr_N}\sqrt{\vr_N}\vu_N \otimes \vu_N \to \Ov{\vr} \ \Ov{\bu \otimes \bu} \ \mbox{weakly in} \ L^r((0,T) \times \Omega; \R^{d \times d}), 
\]
and, consequently,  
\begin{equation} \label{I2}
\int\limits_{0}^{T} \int\limits_{\Omega} (\vr_N \vu_N \otimes \vu_N): \Grad(\vc{R}_N ( \vc{h}_{1,N}, \cdots, \vc{h}_{N,N})[ \bfphi ]) \dx \dt\rightarrow \int\limits_{0}^{T} \int\limits_{\Omega} \Ov{\vr} \ \overline{(\bu\otimes \bu)}:\Grad \bfphi \dx \dt \mbox{ as }N\rightarrow \infty,
\end{equation}
for any  $\bfphi \in C^1_c([0,T) \times \Omega; \R^d)$, $\Div \bfphi = 0$.} 

{Finally, following step by step the arguments of
\cite[Section 5.2.2]{FeiRoyZar2022} we show 
\[
\intO{ \vr_N \vu_N \cdot \phi } \to \intO{ \Ov{\vr} \vu \cdot \phi }\ 
\ \in L^r(0,T) \ \mbox{for any}\ r \geq 1,\  \mbox{and any}\ \phi \in C^1_c(\Omega),\ \Div \phi = 0.
\]
Consequently, my means of \eqref{con:ws}, \eqref{con:w} and the standard Aubin--Lions argument, 
\begin{equation} \label{l2c}
\vu_N \to \vu \ \mbox{(strongly) in}\ L^2((0,T) \times \Omega; \R^d),
\end{equation}
in particular,}
\begin{equation} \label{c13}
\Ov{ \bu \otimes \bu } =  \bu \otimes \bu.
\end{equation}

\subsubsection{Time derivative}

Our next goal is to establish the limit 
\begin{equation}\label{I3}
\int\limits_{0}^{T} \int\limits_{\Omega} \vr_N \vu_N \cdot\partial_t (\vc{R}_N ( \vc{h}_{1,N}, \cdots, \vc{h}_{N,N})[ \bfphi ]) \dx \dt \to   \int\limits_{0}^{T} \int\limits_{\Omega}  \bu \cdot\partial_t  \bfphi  \dx \dt.
\end{equation}
{Recall the notation, 
\[
\vc{Y}_{n,N}(t) = \frac{\D }{\dt} \vc{h}_{n,N}(t) \ \mbox{for a.a.}\ t \in (0,T). 
\]
}

{First, exactly as in \cite[Section 5.1]{FeiRoyZar2022}, we have the formula}
\begin{align}\label{error:time}
&\partial_t \bfphi - \partial_t \vc{R}_N ( \vc{h}_{1,N}, \cdots, \vc{h}_{N,N})[ \bfphi ] = 
\partial_t \bfphi -  \vc{R}_N ( \vc{h}_{1,N}, \cdots, \vc{h}_{N,N})[ \partial_t \bfphi ]  \br 
&+ \sum_{n=1}^N \nabla_{\vc{h}_{n}} \vc{R}_N ( \vc{h}_{1,N}, \cdots, \vc{h}_{N,N})[ \bfphi ] \cdot \vc{Y}_{n,N} \br
&=\partial_t \bfphi -  \vc{R}_N ( \vc{h}_{1,N}, \cdots, \vc{h}_{N,N})[ \partial_t \bfphi ] \br
&+ \sum_{n=1}^{N} \vc{R}_N (\vc{h}_{1,N}, \cdots, \vc{h}_{n-1,N}) \left[ \Grad \Big( \vc{R}_{5^{n-1} r_{n,N}}(\vc{h}_{n,N}, \cdots, \vc{h}_{N,N}) [\bfphi] \Big) \cdot \vc{Y}_{n,N} \right.
\br &- \left. \vc{R}_{5^{n-1} r_{n,N}}(\vc{h}_{n,N}) \left[ \Grad \Big( \vc{R}_{5^n r_{n,N}}(\vc{h}_{n+1,N}, \cdots, \vc{h}_{N,N}) [\bfphi] \Big) \cdot \vc{Y}_{n,N}                           \right]  \right], 
\end{align}
with the {notational} convention  
\[
\vc{R}_N (\vc{h}_{1,N}, \cdots, \vc{h}_{0,N}) = 
\vc{R}_{5^N r_{N,N}} (\vc{h}_{N+1,N}, \cdots, \vc{h}_{N,N}) = {\rm Id},
\]
{where, in accordance with \eqref{formula1}}
\begin{equation} \label{supp}
{\rm supp} \left[  \sum_{n=1}^N \nabla_{\vc{h}_{n}} \vc{R}_N ( \vc{h}_{1,N}, \cdots, \vc{h}_{N,N})[ \bfphi ] \right] \subset 
\cup_{n=1}^N B_{2 (5^{n-1} r_{n,N})} (\vc{h}_{n,N}).
\end{equation}

Similarly to the proof of \eqref{error1}, {we have} 
\begin{equation} \label {error2}
	\| ( \partial_t \bfphi -  \vc{R}_N ( \vc{h}_{1,N}, \cdots, \vc{h}_{N,N})[ \partial_t \bfphi ] )(t, \cdot) \|_{W^{1,p}(\R^d; \R^d)} \to 0 \ \mbox{as}\ N\to \infty
	\ \mbox{uniformly in}\ t \in [0,T]. 
\end{equation}

{Next, in accordance with \eqref{supp}, the last error term in \eqref{error:time} can be estimated as}
\begin{align} 
&\left| \sum_{n=1}^{N} \vc{R}_N (\vc{h}_{1,N}, \cdots, \vc{h}_{n-1,N}) \Big[ \Grad \vc{R}_{5^{n-1} r_{n,N}}(\vc{h}_{n,N}, \cdots, \vc{h}_{N,N}) [\bfphi] \cdot \vc{Y}_{n,N} \right.
\br &- \left. \vc{R}_{5^{n-1} r_{n,N}}(\vc{h}_{n,N}) \left[ \Grad \vc{R}_{5^n r_{n,N}}(\vc{h}_{n+1,N}, \cdots, \vc{h}_{N,N}) [\bfphi] \cdot \vc{Y}_{n,N}                           \right]  \Big] \right| \br 
& \quad \quad \leq \mathds{1}_{ \cup_{n = 1}^N B_{2( 5^{n-1} r_{n,N})}(\vc{h}_{n,N}) }
\sum_{n=1}^{N} |\vc{Y}_{n,N} | \left(  \Big| \vc{R}_N (\vc{h}_{1,N}, \cdots, \vc{h}_{n-1,N})   \Big[ \Grad \vc{R}_{5^{n-1} r_{n,N}}(\vc{h}_{n,N}, \cdots, \vc{h}_{N,N}) [\bfphi] \Big]  \right. \br &\quad \quad \quad \quad \quad \quad + 
\left. \vc{R}_N (\vc{h}_{1,N}, \cdots, \vc{h}_{n,N}) \Big[ \Grad \vc{R}_{5^n r_{n,N}}(\vc{h}_{n+1,N}, \cdots, \vc{h}_{N,N}) [\bfphi] \Big] \Big|\right)  .
\label{error3} 
\end{align}

In view of the estimates \eqref{error2}, \eqref{error3}, we can establish \eqref{I3} by showing
\begin{equation} \label{I3bis}
\int_0^T \intO{ \mathds{1}_{ \cup_{n = 1}^N B_{2 (5^{n-1} r_{n,N})}(\vc{h}_{n,N})  } \rho_N {|\bu_N |} \sum_{n= 1}^N |\vc{Y}_{n,N}| |\vc{G}_{n,N}|
	} \dt \to 0	\quad\mbox{as}\quad N\rightarrow\infty,
	\end{equation} 
where, 
\begin{multline*}
\vc{G}_{n,N}= \vc{R}_N (\vc{h}_{1,N}, \cdots, \vc{h}_{n-1,N})   \Big[ \Grad \vc{R}_{5^{n-1} r_{n,N}}(\vc{h}_{n,N}, \cdots, \vc{h}_{N,N}) [\bfphi] \Big]  \\ + \vc{R}_N (\vc{h}_{1,N}, \cdots, \vc{h}_{n,N}) \Big[ \Grad \vc{R}_{5^n r_{n,N}}(\vc{h}_{n+1,N}, \cdots, \vc{h}_{N,N}) [\bfphi] \Big].
\end{multline*}
By virtue of \eqref{E14bisR}, \eqref{E14R},
\begin{equation} \label{I4bis}
\left\| \vc{G}_{n,N} \right\|_{L^\infty(0,T; L^r(\Omega; \R^d ))} \leq c(r)^{N}\| \bfphi \|_{L^{\infty}(0,T; W^{1,r} (\R^d; \R^d))}
\ \mbox{for any}\ {1 < r < \infty},\ n=1,\cdots, N.
\end{equation}

{Consequently, by means of H\" older inequality and the Sobolev embedding, we get
\begin{multline} \label{I4}
	\left| \intO{ \mathds{1}_{ \cup_{n = 1}^N B_{2\cdot 5^{n-1} r_{n,N}}(\vc{h}_{n,N})  } \rho_N |\bu_N | \cdot  |\vc{Y}_{n,N}| |\vc{G}_{n,N}|
	} \right| \\
	\leq c(r)^{N}\|\bfphi\|_{W^{1,r}(\Omega)}\|\vr_N\|_{L^q(\Omega)}\|\bu_N\|_{L^{s}(\Omega; \R^d)}|\vc{Y}_{n,N}(t)|\left(\sum\limits_{n=1}^N |B_{2( 5^{n-1} r_{n,N})}|\right)^{1-\frac{1}{r}-\frac{1}{q} - \frac{1}{s}}, 
\end{multline}
where  $q > 1$ and $r$ and $s$ are arbitrary finite.}

{Now, in view of hypothesis \eqref{w10a},
\[
\int_{\mathcal{S}_{n,N}(t) } |\vc{w}_{n,N}|^2 \ \dx = 
\int_{\mathcal{S}_{n,N}(t) } \left( |\vc{Y}_{n,N}(t)|^2 
 + |{\bfomega}_{n,N} \wedge (x - \vc{h}_{n,N}(t) ) |^2 \right) \dx.
\]	
Consequently, following the arguments of \cite[Section 3.1]{FeiRoyZar2022} 
and making use of hypothesis \eqref{w10} as the case may be, we obtain the bound
on the rigid translation velocities
\begin{equation} \label{rigid}
\| \vc{Y}_{n,N} \|_{L^p(0,T)} \aleq |\mathcal{S}_{n,N}|^\alpha , \ \alpha > 0 \ \mbox{arbitrary if} \ p = d, \ \alpha = 0 \ \mbox{if}\ p > d.
\end{equation}
}
Thus we may fix $A = A(p,q)$ in hypothesis \eqref{w14} so that the right--hand side of \eqref{I4} tends to zero in $L^1(0,T)$ as $N \to \infty$ for any fixed 
test function $\bfphi \in C^1_c([0,T) \times \Omega; \R^d)$, $\Div \bfphi = 0$.

\subsection{Identifying the limit in the viscous stress}
\label{s}

{Summarizing the preceding part we
have obtained the system of equations} 
\begin{equation} \label{s1a}
	\int_0^T \intO{ \Big[ \Ov{\vr} \vu \cdot \partial_t \bfphi + \Ov{\vr} (\vu \otimes \vu) : \Grad \bfphi 
		\Big] } \dt = \int_0^T \intO{ \Big[ \mathbb{S}: \Ds \bfphi - \Ov{\vr} \vc{g} \cdot \bfphi 
		\Big] } - \intO{ \vu_0 \cdot \bfphi(0, \cdot) }
\end{equation}
for any $\bfphi \in C^1_c([0,T) \times \Omega; \R^3)$, $\Div \bfphi = 0$, together with the incompressibility constraint 
\begin{equation} \label{s1b}
\Div \vu = 0 \ \mbox{a.a. in}\ (0,T) \times \Omega.
\end{equation}

Our ultimate goal is to identify the weak limit of the viscous stresses. Specifically, we want to show that $$\mathbb{S} \in \partial F(\Ds \vu) \ \mbox{for a.a.}\ 
t \in (0,T).$$ 
As shown in \eqref{con:S},  
\[
\mathbb{S}_N \to \mathbb{S} \ \mbox{weakly in}\ L^{p'}(0,T; L^{p'}(\Omega; \R^d)). 
\]

We recall that 
\[
\vu_N \to \vu \ \mbox{weakly in}\ 
L^p(0,T; W^{1,p}_0(\Omega; \R^d)) \ \mbox{and (strongly) in}\ 
L^2(0,T; L^2(\Omega; \R^d)). 
\]
Moreover $(\vr_N,\vu_N)$ satisfy the energy inequality \eqref{s2}.
In addition, using hypothesis \eqref{w11b} concerning (strong) convergence of the initial data, 
we get 
\[
\lim_{N \to \infty} \frac{1}{2} \intO{ \vr_{0,N} 
	|\vu_{0, N}|^2 } + \int_0^\tau \intO{ \vr_N \vc{g} \cdot \vu_N } \dt = 
\frac{1}{2} \intO{ \Ov{\vr} 
	|\vu_{0}|^2 } + \int_0^\tau \intO{  \Ov{\vr} \vc{g} \cdot \vu } \dt.	
\]
Finally, as convex functions are weakly lower semi--continuous, we may infer that 
\begin{align} 
	\frac{1}{2} \intO{ \Ov{\vr} |\vu|^2(\tau, \cdot) } &+ 
	\int_0^\tau \intO{ \Big[ F(\Ds \vu) + F^* (\mathbb{S}) \Big] } \br &\leq
	 \limsup_{N \to \infty} \left[ 	\frac{1}{2} \intO{ \vr_N |\vu_N|^2(\tau, \cdot) } + 
	 \int_0^\tau \intO{ \Big[ F(\Ds \vu_N) + F^* (\mathbb{S}_N) \Big]  } \right]
	\br &\leq
	 \frac{1}{2} \intO{  
		\Ov{\vr} |\vu_{0}|^2 } + \int_0^\tau \intO{  \Ov{\vr} \vc{g} \cdot \vu } \dt\  
		\mbox{for a.a.}\ \tau \in (0,T).
	\label{s5} 
\end{align}

The next step is to derive the energy balance for the limit system \eqref{s1a}, \eqref{s1b}. 
As $p \geq d$, we have, by the standard Sobolev embedding relations, 
\[
\vu \otimes \vu \in {L^{\frac{p}{2}} (0,T; L^s(\Omega)) \cap L^\infty(0,T; L^2(\Omega; \R^d)) \ \mbox{for any finite}\ s  \geq 1. } 
\]
In particular, 
\[
\vu \otimes \vu \in L^{p'}(0,T; L^{p'}(\Omega; \R^{d \times d}). 
\]
Thus the limit system \eqref{s1a} can be interpreted as 
\[
\partial_t \vu = - \Div (\vu \otimes \vu) + \frac{1}{\Ov{\vr}} \Div \mathbb{S} + \vc{g} 
\in L^{p'}(0, T; L^{p'}(\Omega; \R^d)) + \Grad \Pi,  
\]
where the solenoidal field $\vu$ can be used as ``test function'' (see e.g. von Wahl \cite{vonW}). 
Accordingly, we deduce the energy \emph{equality}
\begin{equation} \label{s6}
\frac{1}{2} \intO{ \Ov{\vr}|\vu (\tau, \cdot)|^2 } + \int_0^\tau \intO{ 
	\mathbb{S}: \Ds \vu } = \frac{1}{2} \intO{ \Ov{\vr}|\vu_0|^2 } + \intO{ \Ov{\vr} \vc{g} \cdot \vu }.
	\end{equation}
Plugging \eqref{s6} in \eqref{s5}, we conclude 
\begin{align}  
\int_0^\tau &\intO{ \Big[ F(\Ds \vu) + F^* (\mathbb{S}) \Big] } \br &\leq
\limsup_{N \to \infty}   
\int_0^\tau \intO{ \Big[ F(\Ds \vu_N) + F^* (\mathbb{S}_N)  \Big] } \leq \int_0^\tau \intO{ \mathbb{S} : \Ds \vu }\dt. 
\ \mbox{for a.a.}\  \tau \in (0,T).	
	\label{s7}
	\end{align}
A direct application of Fenchel's inequality yields the desired conclusion 
\begin{equation} \label{s8}
F(\Ds \vu) + F^* (\mathbb{S}) = \mathbb{S}: \Ds \vu 
\ \Rightarrow \ \mathbb{S} \in \partial F(\Ds \vu) \ \mbox{for a.a.}\ 
t \in (0,T).
\end{equation}		

\subsection{Strong Convergence} 
Finally, we discuss the problem of strong convergence of the velocity gradients. 
It follows from \eqref{s7} that 
\begin{align}
\limsup_{N \to \infty}&   
\int_0^\tau \intO{ \Big[ F(\Ds \vu_N) - F(\Ds \vu) \Big] } + 
\int_0^\tau \intO{ \Big[ F^* (\mathbb{S}_N) - F^*(\mathbb{S})  \Big] }\br &\leq \int_0^\tau \intO{ \Big[ \mathbb{S} : \Ds \vu - F(\Ds \vu) - F^*(\mathbb{S}) \Big]   }\dt \leq 0.
\nonumber
\end{align}
As both $F$ and $F^*$ are convex, we may infer that 
\begin{equation} \label{s9}
\int_0^\tau \intO{ F(\Ds \vu_N) } \to \int_0^\tau \intO{ F(\Ds \vu) },\ 
\int_0^\tau \intO{ F^*(\mathbb{S}_N) } \to \int_0^\tau \intO{ F^*(\mathbb{S}) }. 
\end{equation} 	

Thus if $F$ is strictly convex,
we conclude, passing to a subsequence as the case may be, 
\[
\Ds \vu_N \to \Ds \vu \ \mbox{a.a. in}\ (0,T) \times \Omega, 
\]
yielding (for the original sequence)
\begin{equation} \label{s10}
\Ds \vu_N \to \Ds \vu \ \mbox{in}\ L^1((0,T) \times \Omega; \R^{d \times d}_{\rm sym}), 
\end{equation}
cf. \cite[Appendix, Theorem 11.27]{FeNo6A}.

Finally, as $F$ complies with the growth restriction \eqref{i5}, we combine 
\eqref{s9}, \eqref{s10} to conclude 
\[
\int_0^T \intO{ 
	|\Ds \vu_N |^p }\dt \to \int_0^T \intO{ 
	|\Ds \vu |^p }\dt \ \Rightarrow \ \Ds \vu_N \to \Ds \vu \ \mbox{in}\ 
	L^p((0,T) \times \Omega; R^{3 \times 3}_{\rm sym}) . 
\]

\section{Concluding remarks}\label{c}

We conclude by several remarks concerning possible extensions of the above result.

\subsection{Total volume}

If $d = p = 2$, the hypothesis \eqref{w14} is compatible with the so--called critical distribution of fixed balls 
in the homogenization problem, see e.g. Allaire \cite{Allai4}, while it is super--critical if $p > d$, where any compact set has non--zero $p-$capacity. 
We actually conjecture that the conclusion of Theorem \ref{Tw1} remains if  
\[
{ {\rm vol}[N] \approx A^{-N} \ \mbox{for \emph{some}}\ A > 1.}
\]
The value of the exponential factor $A = A(p,q)$ in Theorem \ref{Tw1} results from the specific construction of the restriction operator $\vc{R}$ that must account for possible collisions of 
several rigid bodies. This stumbling block could be possibly removed if a qualitative information on the mutual distance of two bodies were available.

\subsection{Heavy bodies}

Hypotheses \eqref{w9}, \eqref{w14} allow for ``heavy bodies'', for which the mass densities are allowed to grow as $N \to \infty$. Indeed it is easy to check 
that \eqref{w9} follows from \eqref{w14} {if  
\[
0 < \vr_{n,N} \leq N^\beta ,\ n =1,\dots, N,\ \beta > 0 \ \mbox{arbitrary}.
\]
}
\subsection{More general viscous stresses}

We deliberately focused on shear thickening fluids, more precisely we consider $p \geq d$. Similar results can be expected for general $p > 1$ at least on a short time interval $[0, T_{\rm max})$. Here $T_{\rm max} > 0$ is the life--span of the smooth solution of the \emph{limit} problem \eqref{w13}. Indeed using the 
same method one could establish convergence to generalized dissipative solutions introduced in \cite{AbbFei2} and use the weak--strong uniqueness result proved 
therein.

\def\cprime{$'$} \def\ocirc#1{\ifmmode\setbox0=\hbox{$#1$}\dimen0=\ht0
	\advance\dimen0 by1pt\rlap{\hbox to\wd0{\hss\raise\dimen0
			\hbox{\hskip.2em$\scriptscriptstyle\circ$}\hss}}#1\else {\accent"17 #1}\fi}


\end{document}